\documentclass[11pt]{article}

\usepackage[T1]{fontenc}
\usepackage{lmodern}
\usepackage{amsmath,amssymb,amsthm}
\usepackage[margin=1in]{geometry}
\usepackage{cite}
\usepackage[hidelinks]{hyperref}

\hypersetup{
  pdftitle={The Marcus--Minc Transform Inequality},
  pdfauthor={Yair Lavi}
}

\newtheorem{theorem}{Theorem}[section]
\newtheorem{proposition}[theorem]{Proposition}
\newtheorem{lemma}[theorem]{Lemma}

\title{The Marcus--Minc Transform Inequality}
\author{Yair Lavi}
\date{24 September 2026}

\begin{document}
\maketitle

\begin{abstract}
We prove the Marcus--Minc conjecture: if $n\ge2$ is an integer,
$A$ is a nonnegative $n\times n$ matrix whose row and column sums equal one,
and $J_n$ has every entry $1/n$, then
\[
\operatorname{per}A\ge
\operatorname{per}\!\left(\frac{nJ_n-A}{n-1}\right).
\]
We also determine all equality cases.
\end{abstract}



\section{Introduction and main results}

For an integer $n\ge2$, put $[n]=\{1,\ldots,n\}$, and let $S_n$ be the set
of all permutations of $[n]$. For an $n\times n$ matrix $D=(d_{ij})$, its
permanent is
\[
\operatorname{per}D=\sum_{\pi\in S_n}\prod_{i=1}^n d_{i,\pi(i)}.
\tag{1.1}
\]
A real matrix is \textbf{doubly stochastic} if its entries are nonnegative and the
entries in each row and in each column sum to one. We denote the set of
$n\times n$ doubly stochastic matrices by $\Omega_n$. Let $\mathbf1$ be the
all-ones column vector, let
\[
J_n=\frac1n\mathbf1\mathbf1^{\mathsf T},
\]
and define the Marcus--Minc transform $\tau_n:\Omega_n\to\Omega_n$ by
\[
\tau_n(A)=\frac{nJ_n-A}{n-1}.
\tag{1.2}
\]
The map is well defined: every entry of a doubly stochastic matrix is at most
one, so $nJ_n-A$ is nonnegative, and each of its row and column sums is
$n-1$. A \textbf{permutation matrix} is a zero--one matrix having exactly one $1$
in each row and column.

In 1967 Marcus and Minc \cite{ref1} conjectured that
$\operatorname{per}A\ge\operatorname{per}\tau_n(A)$ for every $A\in\Omega_n$,
and that for $n\ge3$ equality holds only at $J_n$; we call this the
Marcus--Minc conjecture. As reported in \cite{ref7,ref9}, they proved the inequality
for $n=2$, for symmetric positive semidefinite $A$, and for $A$ sufficiently
close to $J_n$. Wang proved the inequality in order three and showed that
the original equality claim fails \cite{ref2}: equality also holds for the six matrices
$(\mathbf1\mathbf1^{\mathsf T}-P)/2$, with $P$ a permutation matrix.
Foregger proved it in order four \cite{ref3}. Hwang proved it for partly decomposable
matrices \cite{ref4}, and Malek proved it near $J_n$ and under a spectral-sector
condition \cite{ref5,ref6}. Cheon and Wanless recorded the general assertion as
Conjecture 3 in their 2005 update of Minc's catalogue \cite{ref6}. The problem also
appears in later work \cite{ref7,ref8,ref9}.

Our first result proves the Marcus--Minc conjecture and classifies equality in
every dimension.

\begin{theorem}
For every integer $n\ge2$ and every $A\in\Omega_n$,
\[
\operatorname{per}A\ge\operatorname{per}\tau_n(A).
\tag{1.3}
\]
The equality cases are exactly the following:

\begin{enumerate}
\item if $n=2$, every $A\in\Omega_2$;
\item if $n=3$, the matrix $J_3$ and the six matrices
$(\mathbf1\mathbf1^{\mathsf T}-P)/2$, with $P$ a permutation matrix;
\item if $n\ge4$, only $J_n$.
\end{enumerate}
\end{theorem}

We also give new proofs for the known cases $n=3,4$; $n=2$ is immediate.

The proof comes from two estimates with different dimensional scales. To state
them, define the Frobenius norm of $X=(x_{ij})$ by
\[
\|X\|_F=\left(\sum_{i,j}x_{ij}^2\right)^{1/2},
\]
and set
\[
d_n=\frac{n!}{n^n},\qquad
r^2=\|A-J_n\|_F^2,\qquad
a_n=\frac1{2n}+\frac1{3n^2}.
\tag{1.4}
\]
For real $t$, also put
\[
T_n(t)=\sum_{k=0}^n\frac{(nt)^k(1-t)^{n-k}}{k!},
\qquad
\beta_n=\frac{T_n(1/(n-1))-1}{n-1}.
\tag{1.5}
\]
\begin{theorem}
For every $n\ge2$ and every $A\in\Omega_n$,
\[
\frac{\operatorname{per}A}{d_n}\ge e^{a_nr^2},
\qquad
\frac{\operatorname{per}\tau_n(A)}{d_n}\le1+\beta_nr^2.
\tag{1.6}
\]
\end{theorem}

For $n\ge5$ one has $\beta_n<1/(2n)<a_n$. In order four,
\[
\frac18<\beta_4=\frac{31}{243}<a_4=\frac7{48},
\qquad
a_4-\beta_4=\frac{71}{3888}.
\tag{1.7}
\]
Thus order four genuinely uses the $1/(3n^2)$ term in $a_n$; the simpler
coefficient $1/(2n)$ does not separate the two bounds there.

The bounds also give a global quantitative form of Theorem 1.1. Define
\[
G_n(A)=\operatorname{per}A-\operatorname{per}\tau_n(A)
\]
and, for $n\ge5$,
\[
\epsilon_n=\frac1{2n}-
\frac{n}{(n-1)(2(n-2)^2-n)},
\qquad
\epsilon_4=\frac{71}{3888}.
\tag{1.8}
\]
We will prove that every $\epsilon_n$ in (1.8) is positive. Consequently,
\[
G_n(A)\ge d_n\epsilon_n\|A-J_n\|_F^2
\qquad(n\ge4).
\tag{1.9}
\]
The constants in (1.9) are explicit rather than asserted to be optimal. The
sharper coefficient supplied directly by (1.6) is $a_n-\beta_n$.

Section 2 proves the lower estimate in (1.6). Sections 3 and 4 prove the
transformed upper estimate. Section 5 separates the two bounds for every
$n\ge4$. Section 6 treats orders three and two directly and determines all
equality cases.

\section{The global lower bound}

\subsection{A coefficient inequality}

A polynomial $p\in\mathbb R[x_1,\ldots,x_m]$ is \textbf{real stable} if either
$p$ is the zero polynomial or
\[
p(z_1,\ldots,z_m)\ne0
\quad\text{whenever}\quad
\operatorname{Im}z_1,\ldots,\operatorname{Im}z_m>0.
\]
It is \textbf{homogeneous of degree $m$} if every monomial with nonzero coefficient
has total degree $m$. For a homogeneous polynomial of degree $m$ in $m$
variables with nonnegative coefficients, define its capacity by
\[
\operatorname{Cap}(p)=
\inf_{x_1,\ldots,x_m>0}
\frac{p(x_1,\ldots,x_m)}{x_1\cdots x_m}.
\tag{2.1}
\]
The notation $[x_1\cdots x_m]p$ means the coefficient of the monomial
$x_1\cdots x_m$ in $p$. We use the following coefficient bound of Gurvits
\cite[Theorem~2.4 and Corollary~2.5]{ref10}. A proof is included because the
zero-capacity boundary case requires care and because the bound is the only
stable-polynomial input to our argument.

\begin{proposition}[Gurvits]
If $p$ is real stable, homogeneous of degree
$m$ in $m$ variables, and has nonnegative coefficients, then
\[
[x_1\cdots x_m]p\ge\frac{m!}{m^m}\operatorname{Cap}(p).
\tag{2.2}
\]
\end{proposition}

\begin{proof}
If $p=0$ or $m=1$, the claim is immediate; hence assume $p\ne0$
and $m\ge2$. We first
recall two closure facts used below. Specializing one variable of a real-stable
polynomial at a real value preserves real stability or gives zero: approach
that value from the upper half-plane and apply Hurwitz's theorem. Partial
differentiation has the same property. Indeed, if $d=\deg_{x_m}p\ge1$, the
leading coefficient $p_d(x_1,\ldots,x_{m-1})$ is nonvanishing when all its
arguments lie in the upper half-plane: the nonvanishing polynomials
$p(z_1,\ldots,z_{m-1},it)/(it)^d$ converge locally uniformly to $p_d$ as
$t\to\infty$, and Hurwitz's theorem applies. Thus
$p(z_1,\ldots,z_{m-1},\cdot)$ is nonconstant there. Gauss--Lucas places all
roots of its derivative in the closed lower half-plane, so the partial
derivative is real stable. If $d=0$, that derivative is zero.

For integers $s\ge2$, put
$g(s)=((s-1)/s)^{s-1}$, and set $g(1)=1$.

Let $h$ be a one-variable polynomial of degree at most $s$, with nonnegative
coefficients and all roots real and nonpositive. We first show
\[
h'(0)\ge g(s)\inf_{t>0}\frac{h(t)}t.
\tag{2.3}
\]
If $h(0)>0$, write
$h(t)=h(0)\prod_{i=1}^s(1+c_it)$ with $c_i\ge0$, adding zero values of $c_i$
when the degree is smaller than $s$. Set
$u=\sum_i c_i=h'(0)/h(0)$. The arithmetic--geometric mean inequality gives
\[
h(t)\le h(0)\left(1+\frac{ut}{s}\right)^s.
\]
If $u>0$ and $s\ge2$, evaluate the right-hand side divided by $t$ at
$t=s/[u(s-1)]$ to obtain (2.3). If $u=0$, both sides of (2.3) are zero; if
$s=1$, (2.3) is the direct identity for a linear polynomial. If $h(0)=0$,
nonnegativity of the coefficients gives
$\inf_{t>0}h(t)/t=h'(0)$, including the case in which $t^2$ divides $h$.

Now fix positive values of $x_1,\ldots,x_{m-1}$ and apply (2.3) to
$h(t)=p(x_1,\ldots,x_{m-1},t)$. Real stability, followed by specialization at
positive real values, makes every root of this one-variable polynomial real;
nonnegative coefficients make every root nonpositive. If
\[
q(x_1,\ldots,x_{m-1})=
\left.\frac{\partial p}{\partial x_m}\right|_{x_m=0},
\]
then the definition of capacity and (2.3) give
\[
q(x_1,\ldots,x_{m-1})
\ge g(m)\operatorname{Cap}(p)x_1\cdots x_{m-1}.
\]
Hence $\operatorname{Cap}(q)\ge g(m)\operatorname{Cap}(p)$. The closure facts
above allow the same step to be iterated. At each stage the symbol $s$ in
$g(s)$ is the number of
variables still present. After the last differentiation the resulting
constant is $[x_1\cdots x_m]p$, and
\[
\prod_{s=2}^m g(s)=\frac{m!}{m^m}.
\]
This proves (2.2), including zero capacity.
\end{proof}

\subsection{A one-column bound}

Fix $A=(a_{ij})\in\Omega_n$ whose entries are initially all positive, and
fix a column index $j$. Write
\[
c_i=a_{ij},
\qquad
L_i(x)=\sum_{k\ne j}a_{ik}x_k.
\]
Consider
\[
p_A(x)=\prod_{i=1}^n\left(\sum_{k=1}^n a_{ik}x_k\right).
\]
Every linear factor is nonzero when all variables have positive imaginary
parts, so $p_A$ is real stable. Differentiation and specialization give the
nonzero real-stable polynomial
\[
q_j(x)=\left.\frac{\partial p_A}{\partial x_j}\right|_{x_j=0}
=\sum_{i=1}^n c_i\prod_{s\ne i}L_s(x).
\tag{2.4}
\]
It is homogeneous of degree $n-1$ in the $n-1$ variables indexed by
$k\ne j$, and the coefficient of $\prod_{k\ne j}x_k$ is
$\operatorname{per}A$.

Since $\sum_i c_i=1$, weighted arithmetic--geometric mean applied to (2.4)
gives
\[
q_j(x)\ge
\prod_{i=1}^n\left(\prod_{s\ne i}L_s(x)\right)^{c_i}
=\prod_{s=1}^n L_s(x)^{1-c_s}.
\tag{2.5}
\]
For each row $s$, the remaining row sum is
$\sum_{k\ne j}a_{sk}=1-c_s$. A second weighted arithmetic--geometric mean
inequality gives
\[
L_s(x)^{1-c_s}
\ge(1-c_s)^{1-c_s}\prod_{k\ne j}x_k^{a_{sk}}.
\tag{2.6}
\]
Multiplying (2.6) over $s$ and using the remaining column sums shows
\[
\operatorname{Cap}(q_j)\ge\prod_{i=1}^n(1-c_i)^{1-c_i}.
\]
Proposition 2.1, with $m=n-1$, therefore yields
\[
\operatorname{per}A\ge
d_{n-1}\prod_{i=1}^n(1-a_{ij})^{1-a_{ij}}.
\tag{2.7}
\]
The full one-column bound (2.7) is implicit in the proof of
Laurent and Schrijver \cite[proof of Corollary~1d, equation~(39)]{ref11}; our use of it is the
explicit quadratic estimate (2.11). For an arbitrary
$A\in\Omega_n$, apply (2.7) to $(1-\delta)A+\delta J_n$ and let
$\delta\downarrow0$. We use the continuous convention $0^0=1$. This proves
(2.7) on the boundary without assuming continuity or attainment of capacity.

\subsection{From one column to the global lower bound}

Define $f:[0,1]\to\mathbb R$ by
\[
f(x)=(1-x)\log(1-x),
\]
with its continuous value $f(1)=0$, and define the squared deviation of
column $j$ from the uniform column by
\[
F_j=\sum_{i=1}^n\left(a_{ij}-\frac1n\right)^2.
\]
Because $d_{n-1}(1-1/n)^{n-1}=d_n$, taking logarithms in (2.7) gives
\[
\log\frac{\operatorname{per}A}{d_n}
\ge\sum_{i=1}^n f(a_{ij})-nf(1/n).
\tag{2.8}
\]
The convergent endpoint-valid expansion
\[
f(x)=-x+\sum_{k\ge2}\frac{x^k}{k(k-1)}
\tag{2.9}
\]
makes the gain over the uniform column explicit. Put
$s_2=\sum_i a_{ij}^2=1/n+F_j$. Jensen's inequality with weights $a_{ij}$
gives
\[
\sum_i a_{ij}^3\ge s_2^2,
\qquad
\sum_i a_{ij}^3-\frac1{n^2}
\ge\frac{2F_j}{n}+F_j^2.
\tag{2.10}
\]
For every $k\ge4$, ordinary convexity gives
$\sum_i a_{ij}^k\ge n^{1-k}$. Substituting (2.9) into the difference in (2.8),
the linear terms cancel, the quadratic term contributes $F_j/2$, and (2.10)
controls the cubic term. Consequently
\[
\sum_i f(a_{ij})-nf(1/n)
\ge\left(\frac12+\frac1{3n}\right)F_j+\frac{F_j^2}{6}
\ge\left(\frac12+\frac1{3n}\right)F_j.
\tag{2.11}
\]
Since $\sum_jF_j=r^2$, some column has $F_j\ge r^2/n$. Equations
(2.8)--(2.11) prove the first inequality in (1.6):
\[
\frac{\operatorname{per}A}{d_n}\ge
\exp\!\left[\left(\frac1{2n}+\frac1{3n^2}\right)r^2\right].
\tag{2.12}
\]
\section{A subpermanent estimate for zero-margin matrices}

For subsets $S,T\subseteq[n]$ of the same size, let $X[S,T]$ be the
submatrix of $X$ with row set $S$ and column set $T$. For $0\le k\le n$,
define the sum of all $k\times k$ subpermanents by
\[
\sigma_k(X)=
\sum_{\substack{S,T\subseteq[n]\\ |S|=|T|=k}}
\operatorname{per}X[S,T],
\qquad \sigma_0(X)=1.
\tag{3.1}
\]
A matrix has \textbf{zero margins} if every row sum and every column sum is zero.
Its Euclidean operator norm is
\[
\|X\|_{2\to2}=\sup_{v\ne0}\frac{\|Xv\|_2}{\|v\|_2}.
\]
Write $I_n$ for the identity matrix and set
\[
Q=I_n-J_n.
\]
Put
$H=\mathbf1^\perp=\{v\in\mathbb R^n:\mathbf1^{\mathsf T}v=0\}$. The map
$Q$ is the identity on $H$ and is zero on the line
$\mathbb R\mathbf1$; this is what we mean by the orthogonal projection onto
$H$.

\begin{lemma}
If $X$ is a real $n\times n$ matrix with zero margins and
$\|X\|_{2\to2}\le1$, then, for
$2\le k\le n$,
\[
|\sigma_k(X)|\le
\frac{\sigma_k(Q)}{n-1}\|X\|_F^2.
\tag{3.2}
\]
\end{lemma}

\begin{proof}
We use tensor powers of $\mathbb R^n$ with the product inner
product
\[
\langle x_1\otimes\cdots\otimes x_k,
y_1\otimes\cdots\otimes y_k\rangle
=\prod_{\ell=1}^k\langle x_\ell,y_\ell\rangle,
\]
extended bilinearly. If $S=\{i_1,\ldots,i_k\}$ is a $k$-element subset of
$[n]$, put
\[
v_S=\frac1{\sqrt{k!}}\sum_{\pi\in S_k}
e_{i_{\pi(1)}}\otimes\cdots\otimes e_{i_{\pi(k)}},
\qquad
u_k=\sum_{|S|=k}v_S,
\tag{3.3}
\]
where $e_1,\ldots,e_n$ are the standard basis vectors. The vectors $v_S$
are orthonormal, and expanding the inner product gives
\[
\langle v_S,X^{\otimes k}v_T\rangle=\operatorname{per}X[S,T].
\tag{3.4}
\]
Zero margins imply $X=QXQ$. Hence, with
$w=Q^{\otimes k}u_k\in H^{\otimes k}$,
\[
\sigma_k(X)=\langle w,X^{\otimes k}w\rangle,
\qquad
\|w\|^2=\sigma_k(Q)\ge0.
\tag{3.5}
\]
The second identity follows because $Q^{\otimes k}$ is an orthogonal
projection. It also supplies the nonnegativity of every $\sigma_k(Q)$ used
below.

We now define the tensor property responsible for the estimate. A linear
operator $R$ on $H$ is \textbf{self-adjoint} if
$\langle Rx,y\rangle=\langle x,Ry\rangle$ for all $x,y\in H$. For
$w\in H^{\otimes k}$ and $1\le\ell\le k$, its $\ell$-th \textbf{one-factor
marginal} is the unique self-adjoint operator $\rho_\ell$ on $H$ satisfying
\[
\operatorname{tr}(\rho_\ell R)=
\left\langle w,
\bigl(I_H^{\otimes(\ell-1)}\otimes R\otimes
I_H^{\otimes(k-\ell)}\bigr)w\right\rangle
\tag{3.6}
\]
for every self-adjoint operator $R$ on $H$; here $\operatorname{tr}$ denotes
the sum of the diagonal entries of an operator in an orthonormal basis. We
call $w$ \textbf{one-factor isotropic} when each $\rho_\ell$ is a scalar multiple
of $I_H$.

To see existence, choose an orthonormal basis $(h_a)$ of $H$ and, after putting
the $\ell$-th factor first, write $w=\sum_a h_a\otimes z_a$, where
$z_a\in H^{\otimes(k-1)}$. The matrix
$(\langle z_a,z_b\rangle)_{a,b}$ is self-adjoint and satisfies (3.6), so it
defines $\rho_\ell$. If two self-adjoint operators satisfy (3.6), their
difference $D$ gives $\operatorname{tr}(D^2)=0$ upon taking $R=D$, and therefore
$D=0$.

Simultaneously permuting the $n$ coordinate labels fixes $u_k$, commutes with
$Q$, and therefore fixes $w$. Thus each $\rho_\ell$ commutes with every
coordinate permutation on $H$. Extend $\rho_\ell$ by zero on
$\mathbb R\mathbf1$. A matrix commuting with every permutation has one
constant on its diagonal and another off its diagonal. Its restriction to
$H$ is therefore scalar. Since
$\operatorname{tr}\rho_\ell=\|w\|^2$, we obtain
\[
\rho_\ell=\frac{\|w\|^2}{n-1}I_H
\qquad(1\le\ell\le k).
\tag{3.7}
\]
This is the asserted isotropy; it is a property of $w$, not an assumption on
$X$.

Choose linear operators $U,V$ on $H$ satisfying
$U^{\mathsf T}U=V^{\mathsf T}V=I_H$ and numbers
$0\le s_1,\ldots,s_{n-1}\le1$ such that
\[
X|_H=UDV^{\mathsf T},
\qquad
D=\operatorname{diag}(s_1,\ldots,s_{n-1}).
\]
Such a factorization is a singular-value decomposition; the displayed
identities say that $U$ and $V$ are orthogonal, and the upper bounds on the
$s_i$ follow from $\|X\|_{2\to2}\le1$. Put
\[
a=(U^{\mathsf T})^{\otimes k}w,
\qquad
b=(V^{\mathsf T})^{\otimes k}w,
\qquad
K=D^{\otimes k}.
\]
The one-factor marginals of both $a$ and $b$ still equal (3.7). A self-adjoint
operator $R$ is \textbf{positive semidefinite} if
$\langle z,Rz\rangle\ge0$ for every $z$. For self-adjoint operators $R,S$,
write $R\preceq S$ when $S-R$ is positive semidefinite. For every diagonal
entry of $K$, the assumptions $k\ge2$ and $0\le s_i\le1$ give
\[
\prod_{\ell=1}^k s_{i_\ell}
\le\left(\prod_{\ell=1}^k s_{i_\ell}^2\right)^{1/k}
\le\frac1k\sum_{\ell=1}^k s_{i_\ell}^2.
\]
Equivalently,
\[
0\preceq K\preceq\frac1k\sum_{\ell=1}^k
I_H^{\otimes(\ell-1)}\otimes D^2\otimes
I_H^{\otimes(k-\ell)}.
\tag{3.8}
\]
Equations (3.6)--(3.8) yield
\[
\langle a,Ka\rangle,\ \langle b,Kb\rangle
\le\frac{\|w\|^2}{n-1}\operatorname{tr}D^2.
\]
Cauchy--Schwarz for the positive semidefinite operator $K$ now gives
\[
\begin{aligned}
|\sigma_k(X)|
&=|\langle a,Kb\rangle|\\
&\le\sqrt{\langle a,Ka\rangle\langle b,Kb\rangle}\\
&\le\frac{\sigma_k(Q)}{n-1}\|X\|_F^2,
\end{aligned}
\]
because $\operatorname{tr}D^2=\|X\|_F^2$. This proves (3.2).
\end{proof}

For later use, we verify the contraction hypothesis. If $A\in\Omega_n$ and
$x\in\mathbb R^n$, Jensen's inequality in each row and the column sums give
\[
\sum_i\left(\sum_j a_{ij}x_j\right)^2
\le\sum_{i,j}a_{ij}x_j^2
=\sum_jx_j^2.
\tag{3.9}
\]
Thus $\|A\|_{2\to2}\le1$. With $B=A-J_n$, the row and column sum identities
also give $B=QAQ$, so $B$ has zero margins and $\|B\|_{2\to2}\le1$.
Lemma 3.1 applies to $B$.

\section{The transformed upper bound}

Let $B=A-J_n$. Expanding the permanent according to the entries chosen from
$tB$ gives, for every real $t$,
\[
\operatorname{per}(J_n+tB)
=\sum_{k=0}^n\frac{(n-k)!}{n^{n-k}}t^k\sigma_k(B).
\tag{4.1}
\]
Indeed, after choosing the $k$ rows and $k$ columns used by $B$, the remaining
$n-k$ rows can be matched to the remaining columns in $(n-k)!$ ways, and each
remaining entry of $J_n$ contributes $1/n$. Since $B$ has zero margins,
$\sigma_1(B)=0$.

Put $q=1/(n-1)$. Formula (1.2) becomes
$\tau_n(A)=J_n-qB$. Applying Lemma 3.1 to every term of (4.1), and bounding
the odd terms in absolute value rather than assuming a sign, gives
\[
\begin{aligned}
\operatorname{per}\tau_n(A)
&\le d_n+\frac{r^2}{n-1}
\sum_{k=2}^n\frac{(n-k)!}{n^{n-k}}q^k\sigma_k(Q)\\
&=d_n+\frac{r^2}{n-1}
\bigl(\operatorname{per}(J_n+qQ)-d_n\bigr).
\end{aligned}
\tag{4.2}
\]
Here $\sigma_1(Q)=0$ because $Q$ has zero margins, and
$\sigma_k(Q)\ge0$ by (3.5). No positivity of an odd-degree coefficient
$\sigma_k(B)$ is being assumed.

To evaluate the last permanent, write
$J_n+tQ=tI_n+(1-t)J_n$. Choosing the $k$ rows in which the identity term is
used gives
\[
\operatorname{per}(J_n+tQ)
=\sum_{k=0}^n \binom{n}{k}t^k(n-k)!
\left(\frac{1-t}{n}\right)^{n-k}.
\tag{4.3}
\]
Dividing (4.3) by $d_n$ gives exactly $T_n(t)$ from (1.5). Therefore (4.2)
is the second inequality in (1.6):
\[
\frac{\operatorname{per}\tau_n(A)}{d_n}\le1+\beta_nr^2.
\tag{4.4}
\]
\section{Separation in orders at least four}

Assume first that $n\ge5$, and retain $q=1/(n-1)$. Formula (1.5) can be
written as the truncated-exponential identity
\[
T_n(q)=(1-q)^n
\sum_{k=0}^n\frac1{k!}
\left(\frac{nq}{1-q}\right)^k.
\tag{5.1}
\]
All omitted terms of the exponential series are positive. Hence
\[
T_n(q)<
\exp\!\left(n\left[\log(1-q)+\frac{q}{1-q}\right]\right).
\tag{5.2}
\]
The expression in brackets satisfies
\[
\log(1-q)+\frac{q}{1-q}
=\int_0^q\frac{t}{(1-t)^2}\,dt
\le\frac{q^2}{2(1-q)^2}.
\tag{5.3}
\]
Set $x=n/[2(n-2)^2]$. Then $0<x\le5/18<1$, and (5.2)--(5.3), together
with $e^x<1/(1-x)$, imply
\[
T_n(q)<e^x<\frac1{1-x},
\qquad
\beta_n<
\frac{n}{(n-1)(2(n-2)^2-n)}.
\tag{5.4}
\]
The rational expression on the right of (5.4) is smaller than $1/(2n)$.
After multiplication by the positive denominators, this is equivalent to
\[
\begin{aligned}
2n^3-13n^2+17n-8
&=2(n-5)^3+17(n-5)^2\\
&\quad+37(n-5)+2>0.
\end{aligned}
\tag{5.5}
\]
This proves $\beta_n<1/(2n)$ and the positivity of $\epsilon_n$ in (1.8)
for every $n\ge5$. If $A\ne J_n$, then $r>0$, so (2.12), (4.4), and (5.4)
give the strict chain
\[
\frac{\operatorname{per}A}{d_n}
\ge e^{r^2/(2n)}
>1+\frac{r^2}{2n}
>1+\beta_nr^2
\ge\frac{\operatorname{per}\tau_n(A)}{d_n}.
\tag{5.6}
\]
At $A=J_n$, the two permanents both equal $d_n$.

For $n=4$, direct evaluation of (1.5) gives
\[
T_4(t)=1+2t^2+\frac83t^3+5t^4,
\qquad
T_4(1/3)=\frac{112}{81}.
\]
Consequently
\[
\beta_4=\frac{31}{243},
\qquad
a_4=\frac7{48},
\qquad
a_4-\beta_4=\frac{71}{3888}>0,
\]
which proves (1.7) and makes the comparison in (1.3) strict away from $J_4$.

Finally, $e^{a_nr^2}\ge1+a_nr^2$. Subtracting the two bounds in (1.6)
therefore gives
\[
G_n(A)\ge d_n(a_n-\beta_n)r^2.
\tag{5.7}
\]
For $n=4$, (5.7) is (1.9). For $n\ge5$, replacing $a_n$ by the smaller
$1/(2n)$ and using (5.4) gives (1.9) with the explicit $\epsilon_n$ from
(1.8). This establishes (1.3) and its equality case for every $n\ge4$.

\section{Orders three and two}

\subsection{Order three}

Let $A\in\Omega_3$, put $B=A-J_3$, and retain
$r^2=\|B\|_F^2$. Direct expansion of a $3\times3$ zero-margin matrix gives
\[
\sigma_2(B)=\frac12r^2,
\qquad
\operatorname{per}B=\frac23\sum_{i,j=1}^3b_{ij}^3.
\tag{6.1}
\]
Apply (4.1) at $t=1$ and $t=-1/2$. Since
$\tau_3(A)=J_3-B/2$, one obtains
\[
G_3(A)=\frac18\left(r^2+6\sum_{i,j=1}^3b_{ij}^3\right).
\tag{6.2}
\]
Consider one row $(x,y,z)$ of $A$. It consists of nonnegative numbers with
$x+y+z=1$. Its contribution to the expression in parentheses in (6.2) is
\[
\begin{aligned}
&\sum_{u\in\{x,y,z\}}(u-1/3)^2
+6\sum_{u\in\{x,y,z\}}(u-1/3)^3\\
&\qquad=2\bigl(1-4(xy+xz+yz)+9xyz\bigr).
\end{aligned}
\tag{6.3}
\]
Reorder the three entries so that $x\ge y\ge z$. The following elementary
form of Schur's identity, using $x+y+z=1$, is
\[
1-4(xy+xz+yz)+9xyz
=(x-y)^2(x+y-z)+z(x-z)(y-z).
\tag{6.4}
\]
Both terms on the right are nonnegative: in particular,
$x+y-z=1-2z\ge1/3$. Thus every row contribution in (6.3) is nonnegative,
which proves $G_3(A)\ge0$.

Equation (6.4) also determines equality. It first forces $x=y$ and then
forces either $z=0$ or $z=x$. Thus each equality row is either
$(1/3,1/3,1/3)$ or a permutation of $(1/2,1/2,0)$. If exactly one or two
rows are half-rows, the half-rows would have to contribute respectively
$1/3$ or $2/3$ to every column after the uniform rows are removed. Their
contribution to each column is a multiple of $1/2$, so this is impossible.
All rows are therefore uniform, giving $J_3$, or all rows are half-rows. In
the latter case every column contains exactly one zero, the zero positions
form a permutation matrix $P$, and
\[
A=\frac{\mathbf1\mathbf1^{\mathsf T}-P}{2}.
\]
Conversely, $J_3$ and each of these six matrices make every row contribution
in (6.3) zero. This proves the order-three statement and its full equality
classification.

\subsection{Order two}

Every matrix in $\Omega_2$ has the form
\[
A=\begin{pmatrix}t&1-t\\1-t&t\end{pmatrix},
\qquad 0\le t\le1.
\]
The transform replaces $t$ by $1-t$, while
\[
\operatorname{per}A=t^2+(1-t)^2.
\]
Hence every $A\in\Omega_2$ is an equality case.

\subsection{Completion of the proof}

\begin{proof}[Proof of Theorem~1.1]
Section 5 proves (1.3) for $n\ge4$, with equality
only at $J_n$. Section 6.1 proves (1.3) in order three, with exactly the seven
equality cases listed in the theorem. Section 6.2 proves that every matrix
in $\Omega_2$ is an equality case.
\end{proof}

\section*{Acknowledgments}

The proof of this conjecture was carried out by GPT-5.6-sol, GPT-6 Astra,
and Claude Fable 5, under the guidance of the author. The author has reviewed the
resulting proof arguments. Responsibility for the final text rests with the
author.

\end{document}